\documentclass[12pt]{article}
\usepackage{amsmath, amssymb, color, amsthm}
\usepackage{graphicx} 
\usepackage{float}
\usepackage{wrapfig}
\usepackage{comment}

\newtheorem{theorem}{Theorem}
\newtheorem{definition}[theorem]{Definition}
\newtheorem{lemma}[theorem]{Lemma}
\newtheorem{corollary}[theorem]{Corollary}
\newcommand{\p}{\mathcal{P}}

\newcommand\blfootnote[1]{%
  \begingroup
  \renewcommand\thefootnote{}\footnote{#1}%
  \addtocounter{footnote}{-1}%
  \endgroup
}

\makeatletter
\def\blfootnote{\xdef\@thefnmark{}\@footnotetext}
\makeatother

\begin{document}
\title{Congruences modulo primes of the Romik sequence related to the Taylor expansion of the Jacobi theta constant $\theta_3$ {{\bf }}
}
\author{Robert Scherer\footnote{Department of Mathematics, University of California, Davis, One Shields Ave., Davis, CA 95616, USA. Email: rscherer@math.ucdavis.edu
}}

\maketitle

\begin{abstract}
Recently, Romik determined in \cite{Romik} the Taylor expansion of the Jacobi theta constant $\theta_3$, around the point $x=1$.   He discovered a new integer sequence, $(d(n))_{n=0}^\infty=1,1,-1,51,849,-26199,\dots$, from which the Taylor coefficients are built, and conjectured that the numbers $d(n)$ satisfy certain congruences. In this paper, we prove some of these conjectures, for example that $d(n)\equiv (-1)^{n+1}$ (mod 5) for all $n\geq 1$, and that for any prime $p\equiv 3$ (mod 4), $d(n)$ vanishes modulo $p$ for all large enough $n$.
\end{abstract}
\blfootnote{}
\blfootnote{}


\section{Introduction}
\subsection{The sequence $(d(n))_{n={0}}^\infty$ and the main result}
In this paper we will prove a list of congruences, modulo certain prime numbers, satisfied by the integer-valued Romik sequence, which is defined below and whose first several terms are given by \[(d(n))_{n=0}^\infty=1,1,-1,51,849,-26199, 1341999, 82018251, 18703396449,\dots\] (see also \cite{OEIS}). Specifically, we will show: 
\begin{theorem}
\label{main_theorem}
\begin{itemize}
\item[(i)]
$d(n)\equiv 1$ (mod 2) for all $n\geq0$,
\item[(ii)]
$d(n)\equiv (-1)^{n+1}$ (mod 5) for all $n\geq 1$, \, and
\item[(iii)]
if $p$ is prime and $p\equiv 3$ (mod 4), then $d(n)\equiv 0$ (mod p) for all $n>\frac{p^2-1}{2}.$
\end{itemize}
\end{theorem}
This proves half of Conjecture 13 (b) in \cite{Romik}, where the sequence $(d(n))$ was first introduced. (The half of the statement that we don't prove is that for primes $p=4k+1$, the sequence $(d(n))_{n={0}}^\infty$ mod $p$ is periodic, although Theorem \ref{main_theorem} is a specific example of this phenomenon in the case $p=5$.)


The sequence $(d(n))$ is defined in terms of the Jacobi theta constant $\theta_3$, which is the holomorphic function defined on the right half-plane by
\begin{equation}\label{eq:theta1}
\theta_3(x)=1+2\sum_{n=1}^\infty e^{-\pi n^2x}\quad \quad (\text{Re}(x)>0).
\end{equation}
$\theta_3$ satisfies the modular transformation identity
\begin{equation}\label{eqn:transform}
\theta_3\left(\frac{1}{x}\right)=\sqrt{x}\,\theta_3(x),
\end{equation}
which implies that the function $\vartheta$ defined on the upper half-plane by $\vartheta(\tau)=\theta_3(-i\tau)$
is a weight $\frac{1}{2}$ modular form with respect to a certain subgroup of SL$_2(\mathbb{Z})$ (see  \cite[Ch.~4]{Z} and \cite[p. 100]{B}). The numbers $d(n)$ arise in the following way.

\begin{definition}[Romik~\cite{Romik}]\label{def}
Define the function $\sigma$ on the unit disk by 
\begin{equation}\label{eq:sigma1}
\sigma(z)=\frac{1}{\sqrt{1+z}}\theta_3\left(\frac{1-z}{1+z} \right),
\end{equation}
and define the sequence $(d(n))_{n=0}^\infty$ by 
\[
d(n)=\frac{\sigma^{(2n)}(0)}{A\Phi^n },
\]
where $\Phi=\frac{\Gamma\left(\frac{1}{4}\right)^8}{128\pi^4}$, and $A=\theta_3(1)=\frac{\Gamma\left(\frac{1}{4}\right)}{\sqrt{2}\pi^{3/4}}$.
\end{definition}
Thus, the numbers $(d(n))_{n=0}^\infty$ are the Taylor coefficients, modulo trivial factors, of $\sigma$ at $0$. 
It's not at all clear from the definition that the numbers $d(n)$ are integers, but this is shown to be true in $\cite{Romik}$. 
Furthermore, the connection of the sequence $(d(n))$ to the derivatives of $\theta_3$ at $1$ can be made explicit:
\begin{theorem}[Romik~\cite{Romik}]\label{thm:sigma1}
For all $n\geq 0$, 
\begin{equation}\label{eq:theta2}
\theta_3^{(n)}(1)=A\cdot \frac{(-1)^n}{4^n}\sum_{k=0}^{\lfloor{n/2}\rfloor}\frac{(2n)!(4\Phi)^k}{2^{n-2k}(4k)!(n-2k)!} d(k).
\end{equation}
\end{theorem}

\subsection{Taylor coefficients of modular forms}

The results in this paper, which describe congruence properties of a specific integer sequence, may be viewed in the broader context of the study of arithmetic properties of Taylor coefficients of half-integer weight modular forms around complex multiplication points. Given a modular form $f$ of weight $k$, defined on the upper-half plane $\mathbb{H}$, and a CM point $z\in\mathbb{H}$; it is convenient (see e.g. Ch. 5.1 in \cite{Z}) to define the Taylor expansion of $f$ in terms of a new variable $w$ in the unit disk as
\[
(1-w)^{-k}f(M(w))=\sum_{n=0}^\infty c(n) \frac{w^n}{n!} \quad (|w|<1),
\]
where $M$ is the M\"obius transformation given by $M(w)=\frac{z-\bar{z}w}{1-w}$ that maps the unit disk to $\mathbb{H}$, with $M(0)=z$. 
In the case of the modular form $\vartheta(\tau)=\theta_3(-i\tau)$, the Taylor coefficients in the above expansion, about the CM point $z=i$, are the sequence $(d(n))$ (after a proper normalization) introduced in \cite{Romik} and studied here.

The congruences for $(d(n))$ that we consider are analogous to known congruences in the integer weight case. For example it was shown in $\cite{L}$ that if a prime $p$ is inert in the CM-field generated by $z$ and $m\geq1$ is an integer, then the Taylor coefficients eventually vanish modulo $p^m$. However, similar results for weight $\frac{1}{2}$ do not appear to have been established. In addition, the periodicity established here for $p=5$ has analogues for integer weight modular forms, which are known in general cases to have periodic coefficients modulo $p^m$ when $p$ is a split prime (see e.g. \cite{Dats}). We hope that the results proven here for $\theta_3$ will eventually give way to more general statements about the periodicity of congruences of Taylor coefficients of half-integer weight modular forms.

\subsection{An auxiliary matrix and a recurrence for $d(n)$}
In this subsection we recall from \cite{Romik} a recurrence relation for $(d(n))$ in terms of a certain infinite matrix.
\begin{definition}\label{rec}
Define the sequences $(u(n))_{n=0}^\infty$ and $(v(n))_{n=0}^\infty$ by $u(0)=v(0)=1$ and the following recurrence relations for $n\geq 1$:
\begin{equation}\label{eq:u_recursion}
u(n)=\left(3\cdot 7\cdots (4n-1)\right)^2 - \sum_{m=0}^{n-1} {2n+1\choose 2m+1}\left(1\cdot 5\cdots (4(n-m)-3)\right)^2u(m)
\end{equation}
\begin{equation}\label{eq:v_recursion}
\hspace{-2.415cm}v(n)=2^{n-1}\left(1\cdot 5\cdots (4n-3)\right)^2 -\frac{1}{2} \sum_{m=1}^{n-1} {2n\choose 2m}v(m)v(n-m).
\end{equation}
\end{definition}

\begin{definition}
Define the array $(s(n,k))_{1\leq k\leq n}$, by
\begin{equation}\label{eq:r}
s(n,k)=\frac{(2n)!}{(2k)!}[z^{2n}]\left(\sum_{j=0}^\infty \frac{u(j)}{(2j+1)!} z^{2j+1}\right)^{2k},
\end{equation}
where $[z^n]f(z)=[z^n]\sum_{n=0}^\infty c_nz^n$ denotes the $n^{th}$ coefficient $c_n$ in a power series expansion for $f$. 
Also, for $1\leq k\leq n$, define $r(n,k):=2^{n-k}s(n,k).$
\end{definition}
The quantities $r(n,k)$ were introduced in $\cite{Romik}$, along with the following recurrence relation for $d(n)$ (which is used in \cite{Romik} to prove that $d(n)\in\mathbb{Z}$).
\begin{theorem}[Romik~\cite{Romik}]\label{thm:r}
For all pairs $(n,k)$, $1\leq k\leq n$, both $r(n,k)$ and $s(n,k)$ are integers. 
Furthermore, with $d(0)=1$, the following recurrence relation holds for all $n\geq 1$: 
\begin{equation}\label{eq:recurrence_d}
d(n)=v(n)-\sum_{k=1}^{n-1} r(n,k) d(k).
\end{equation}
\end{theorem}

The proof of Theorem \ref{main_theorem} will be based on \eqref{eq:recurrence_d} and a new formula for $s(n,k)$, given in Theorem \ref{thm:partitions_def} below. This formula will provide, among other things, an argument different from the one in \cite{Romik} that $s(n,k)\in\mathbb{Z}$. 

\subsection{Structure of the paper}
In the next section we will derive a formula for $s(n,k)$ mod $p$ that will be an important tool in the rest of the paper. 
In Section \ref{section:odd}, we prove Theorem \ref{main_theorem}, part (i). In Sections \ref{section:p=5} and \ref{section:3mod4} we will give proofs of parts (ii) and (iii), respectively, based on the expression for $s(n,k)$ mod $p$ derived in Section \ref{section:formula},
 the recursive definition \eqref{eq:recurrence_d} for $d(n)$, and a few more facts about the congruences of $(u(n))$ and $(v(n))$. 

\section{A formula for $s(n,k)$ mod $p$}\label{section:formula}
Before we derive the formula, we briefly recall some standard definitions regarding integer partitions. 
 
Let $n$ and $k$ be positive integers. By an \emph {unordered partition $\lambda$ (of $n$ with $k$ parts)} we mean, as usual, a tuple of positive integers, $
\lambda=(\lambda_1,\lambda_2,\dots, \lambda_k)$, with $\lambda_i\leq \lambda_{i+1}$ for $1\leq i<k$, such that $\sum_{i=1}^k \lambda_i = n.$ The numbers $\lambda_i$ are the \emph{parts}. 
We let $\mathcal{P}_{n,k}$ denote the set of unordered partitions of $n$ with $k$ parts, and we let $\p_{n,k}'\subset \p_{n,k}$ be the set of such partitions whose parts are odd numbers.
For a given $\lambda\in \p_{n,k}$, we will let $c_i$ denote the number (possibly $0$) of parts of $\lambda$ whose value is $i$, for $1\leq i\leq n$. 
Thus, the tuple $c(\lambda)=(c_1,c_2,\dots, c_n)$ gives an alternative description of $\lambda$, which we will use freely. (Although each $c_i$ depends on $\lambda$, we choose not to reflect this dependence in the notation, in order to keep it simple, and since it will always be clear from context.)
Finally, observe that $\sum_{i=1}^n ic_i=n$, and $\sum_{i=1}^n c_i=k,$ for each $\lambda\in \p_{n,k}$.

\begin{lemma}[{\cite[pp. 215-216]{Andrews}}]
\label{N}
For any pair $(n,k)$ of positive integers such that $n\geq k$, and any partition $\lambda\in \p_{n,k}$, the number 
\[
\frac{n!}{\prod_{i=1}^{n} i!^{c_i}c_i!}
\]
is an integer. 
\end{lemma}
\noindent\emph{Remark:} The theorem in \cite{Andrews} is the stronger statement that if 
$S$ is a set with $n$ elements, then $\frac{n!}{\prod_{i=1}^{n} i!^{c_i}c_i!}$ is the number of set partitions of $S$ into $k$ blocks $B_i$, with $|B_i|\leq |B_{i+1}|$ for $1\leq i<k$, such that $|B_i|=\lambda_i.$ 

\begin{theorem}
\label{thm:partitions_def}
For any pair $(n,k)$ of positive integers such that $n\geq k$, we have 
\begin{equation}
\label{eq:partitions_def}
s(n,k)=\sum_{\lambda\in \p'_{2n,2k}}\left[\frac{(2n)!}{\prod_{i=1}^{2n} i!^{c_i}c_i!}
\,\prod_{i=1}^{2n}u\left(\frac{i-1}{2}\right)^{c_i}\,\right].
\end{equation}
\end{theorem}

\begin{proof}
We first observe that $\p'_{2n,2k}\neq \emptyset$, since if $n>k$ then $\p'_{2n,2k}$ contains the partition $\lambda$ such that $c(\lambda)$ has $c_1=2k-1, c_{2n-2k+1}=1,$ and $c_i=0$ for all other $i$; while if $n=k$, then $\p'_{2n,2k}$ contains $\lambda$ with $c(\lambda)=(2k,0,\dots, 0)$. 
From \eqref{eq:r} we see that 
\begin{align}
\label{group}
s(n,k)&=\frac{(2n)!}{(2k)!}[z^{2n}]\left(\sum_{\substack{j\geq 1 \\ j \text{ odd}}}\frac{u\left(\frac{j-1}{2}\right)}{j!} z^j\right)^{2k}\nonumber\\
&=\frac{(2n)!}{(2k)!}\sum_{(j_1,j_2,\dots, j_{2k})}
\,\,\prod_{i=1}^{2k}\frac{u\left(\frac{j_i-1}{2}\right)}{j_i!},
\end{align}
where the sum runs over all tuples $j=(j_1,j_2\dots, j_{2k})$ of positive odd integers such that $\sum_{i=1}^{2k}j_i=2n$ (in other words, over all \emph{ordered partitions} of $2n$ into $2k$ odd parts). Call the set of such tuples $\Lambda_{2n,2k}.$ Since $\p'_{2n,2k}$ is nonempty, so is $\Lambda_{2n,2k}$. To each $j\in\Lambda_{2n,2k}$ we associate the unique unordered partition $\lambda\in \p'_{2n,2k}$ obtained by ordering the $j_i$'s in non-decreasing order, and we also associate the tuple $c(\lambda)$. We can define an equivalence relation on $\Lambda_{2n,2k}$ by calling $j$ and $j'$ equivalent if they map to the same $c(\lambda)$ under this association. If $j$ maps to $c(\lambda)=(c_1,\dots, c_{2n})\in \p'_{2n,2k}$, then it is elementary to count that the size of the equivalence class of $j$ is $\frac{(2k)!}{\prod_{i=1}^{2n} c_i!}.$ Furthermore, the product 
$\prod_{i=1}^{2k}\frac{u\left(\frac{j_i-1}{2}\right)}{j_i!}$
in \eqref{group}, as a function of $(j_1,\dots, j_{2k})$, is constant on equivalence classes, and the equivalence classes are indexed by $\p'_{2n,2k}$ in the obvious way. 
Thus, we may rewrite \eqref{group} as
\begin{equation*}
s(n,k)=\frac{(2n)!}{(2k)!}\sum_{\lambda\in\p'_{2n,2k}}
\,\,\left(\frac{(2k)!}{\prod_{i=1}^{2n} c_i!}\,\prod_{i=1}^{2n}\frac{u\left(\frac{i-1}{2}\right)^{c_i}}{i!^{c_i}}\right),
\end{equation*}
which simplifies to \eqref{eq:partitions_def}.
\end{proof}

Henceforth, if $x\in\mathbb{Z}$, and $p\geq 2$ is prime, we let $x_p$ denote the congruence class of $x$ mod $p$.
In light of Lemma \ref{N} and the fact that each $u(n)$ is an integer, we see from \eqref{eq:partitions_def} that $s(n,k)$ is always an integer.
More specifically, each summand in \eqref{eq:partitions_def} 
is a product of integers; therefore we may reduce each summand in \eqref{eq:partitions_def} mod $p$ in the following way to obtain a formula for $s(n,k)_p$. 

\begin{corollary}
\label{cor:partitions_def}
For any pair $(n,k)$ of positive integers such that $n\geq k$, and any prime number $p$, we have
\begin{equation}
\label{eq:partitions_def_p}
s(n,k)_p=\sum_{\lambda\in \p'_{2n,2k}}\left( \left[\frac{(2n)!}{\prod_{i=1}^{2n} i!^{c_i}c_i!}\right]_p \prod_{i=1}^{2n}\left[u\left(\frac{i-1}{2}\right)^{c_i}\right]_p\right),
\end{equation}
where the multiplication in parentheses is of congruence classes, as is the summation over $\p'_{2n,2k}$.
\end{corollary}
\section{Proof of Theorem \ref{main_theorem} (i)}\label{section:odd}
In the previous section we saw that $s(n,k)=\frac{r(n,k)}{2^{n-k}}$ is an integer for all $1\leq k\leq n$, which immediately implies that $r(n,k)$ is even. Thus, \eqref{eq:recurrence_d} implies that in order to show that $d(n)$ is odd for all $n$, it suffices to show that $v(n)$ is odd for all $n$. We will prove this by induction. The first few values of $v(n)$ are given by $(v(n))_{n=0}^\infty = 1,1,47, 7395,\dots,$ which can easily be computed. This establishes the base case. 

Assume now the induction hypothesis that $v(m)$ is odd for all $1\leq m<n$.
We will write $A\equiv B$, for $A,B\in\mathbb{Z}$, to mean that $A$ and $B$ have the same parity. 
We apply the induction hypothesis to simplify the expression in \eqref{eq:v_recursion}, obtaining
\begin{equation}\label{eq:tanay}
v(n)\equiv \frac{1}{2}\sum_{m=1}^{n-1}{2n\choose 2m}=\frac{1}{2}\left [\left(\sum_{m=0}^n {2n\choose 2m}\right)-2\right].\end{equation}
Since 
\begin{align*}
\sum_{m=0}^n{2n\choose 2m}&=\frac{1}{2}\left [\left(\sum_{m=0}^{2n} {2n \choose m}\right) + \sum_{m=0}^{2n} \left({2n\choose m}(-1)^m\right)\right]\\
&=\frac{1}{2}[2^{2n}+0],\\
\end{align*}
we see from \eqref{eq:tanay} that $v(n)\equiv 2^{2n-2}-1\equiv 1,$ as was to be shown.


\section{The behavior of $d(n)$ modulo $p=5$}\label{section:p=5}
\subsection{A formula for $r(n,k)$ mod 5}
Corollary \ref{cor:partitions_def} provides a flexible way to reduce $s(n,k)$ (and hence $r(n,k)$) modulo $p$, and will be our main tool, along with the recurrence relation \eqref{eq:recurrence_d}, in studying the congruences of $d(n)$ modulo primes $p\neq 2$. In the case $p=5$, the reduction \eqref {eq:partitions_def_p} is particularly simple.  Throughout this section the notation $A\equiv B$ will be shorthand for $A\equiv B$ (mod $5$). 
\begin{theorem}[Formula for $r(n,k)$ mod 5]
\label{thm:formula}
For $1\leq k\leq n\leq 5k$ the following congruences hold mod $5$: 
\begin{equation}
\label{eq:r-cases}
r(n,k)\equiv
\begin{cases}
\frac{(2n)!}{\left(\frac{5k-n}{2}\right)! \left(\frac{n-k}{2}\right)!5^\frac{n-k}{2} }& \text{ if }\, n-k \text{ is even}\\[2em]
\frac{2(2n)!}{\left(\frac{5k-n-1}{2}\right)! \left(\frac{n-k-1}{2}\right)!5^\frac{n-k-1}{2}}& \text{ if }\, n-k \text{ is odd}\\
\end{cases}
\end{equation}
If $n>5k,$ then $r(n,k)\equiv 0$. 
\end{theorem}
A graphical plot of Theorem \ref{thm:formula} shows a compelling fractal pattern (see Figure 1 below).

\begin{figure}[h]
\center{\,\,\quad\quad\includegraphics[scale=0.65]{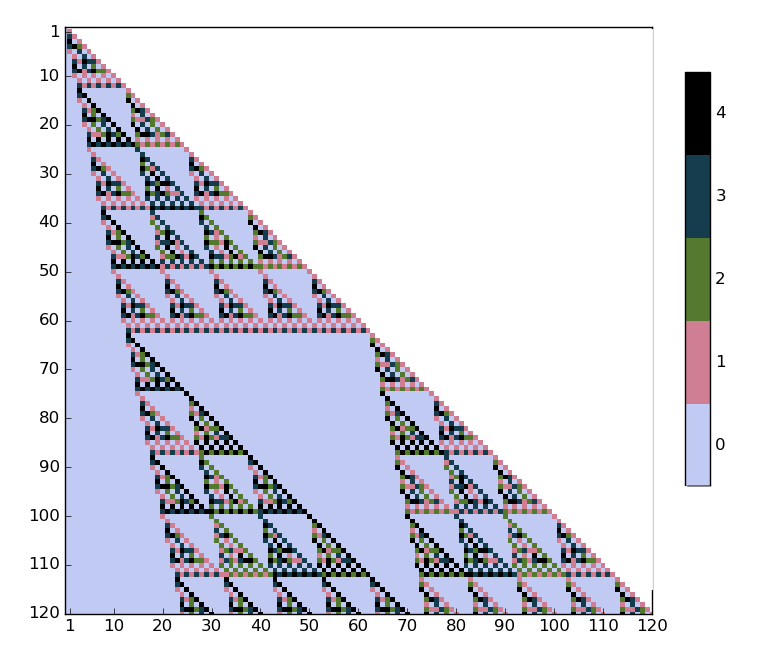}}
\caption{Congruences of $r(n,k)$ mod 5, $1\leq k\leq n< 120$. The rows are indexed by $n$, the columns are indexed by $k$, and the colors indicate residue classes of $r(n,k)$ mod 5, according to the colorbar.}
\end{figure}

Before we begin the proof, we need another lemma about the sequences $(u(n))$ and $(v(n)).$
\begin{lemma}
\label{lemma:uv} 
The sequences $u$ and $v$ satisfy the following congruences mod $5$: 
\begin{align*}
&(i)&(v(n))_{n=0}^\infty \,\,&\equiv (1,1,2, 0, 0, 0, 0,\dots)\\[8 pt]
&(ii)&(u(n))_{n=0}^\infty \,\,&\equiv (1,1,1, 0, 0, 0, 0,\dots)\\
\end{align*}
\end{lemma}
\begin{proof}
From Definition \ref{rec}, one can calculate by hand or with a computer that the first few terms of the sequence $(u(n))_{n=0}^\infty$ are $1,6,256,28560,6071040.$
Furthermore, it is clear from $\eqref{eq:u_recursion}$ that if $n\geq 4$ we have the simplified recursion 
\[
u(n)\equiv {2n+1\choose 2n-1} u(n-1), 
\]
since the term $(3\cdot 7\cdot 11\cdots (4n-1))^2)$ vanishes, as do all of the terms in the summation except for the term corresponding to $m=n-1$. Then by induction we see that $u(n)\equiv 0$ for all $n\geq4$.

Similarly the initial terms of the sequence $(v(n))_{n=0}^\infty$ are $1,1,47, 7395,$ $2453425, 1399055625.$
For $n\geq 2$ the following congruence holds:
\[
v(n)\equiv -\frac{1}{2}\sum_{m=1}^{n-1}{2n\choose 2m}v(m)v(n-m). 
\]
So if we assume that $v(k)\equiv 0$ for $2\leq k \leq n$, then it is clear that $v(n+1)\equiv 0$, and the lemma follows by induction.
\end{proof}
\noindent \emph{Remark:}
Whereas Theorem \ref{thm:partitions_def} is general for all primes, the lemma we just proved was stated for $p=5$. In fact,
experimental evidence suggests that this lemma can be generalized to the statement that $u(n)$ and $v(n)$ are both congruent to $0$ mod $p$ for all $n\geq \frac{p+1}{2}$, when $p$ is a prime congruent to $1$ mod 4. Indeed our proof in the case $p=5$ is rather ad hoc and in particular makes no use of the binomial coefficients appearing in the recursions for $u$ and $v$. In Section 4 we will prove a similar lemma for $u$ and $v$ in the case $p\equiv 3\, (\text{mod } 4).$
\bigskip

\noindent \emph {Proof of Theorem \ref{thm:formula}.}\,\,
In view of Lemma \ref{lemma:uv}, we may restrict the class of partitions that need to be considered in the summation appearing in \eqref{eq:partitions_def_p}. More specifically, let $\mathcal{P}^{3}_{n,k}\subset \mathcal{P}'_{n,k}$ be the set of partitions of $n$ into $k$ parts among the first three odd positive integers, $1,3,5$. Since $u(n)$ vanishes mod $5$ for $n>2$, and therefore $u\left(\frac{i-1}{2}\right)$ vanishes mod $5$ for $i>5$, summands in \eqref{eq:partitions_def_p} that are indexed by partitions not in $\mathcal{P}^{3}_{2n,2k}$ have a residue of $0$ mod $5$. Thus we obtain an equivalent definition of $s(n,k)_5$ to that in \eqref{eq:partitions_def_p} if we replace the indexing set with $\mathcal{P}^3_{2n,2k}$ and adopt the convention that $s(n,k)_5=0$ for pairs $(n,k)$ such that $\mathcal{P}^3_{2n,2k}$ is empty.

Furthermore, for $n=0,1,2$, $u(n)\equiv 1$. Hence, $u\left(\frac{i-1}{2}\right)\equiv 1$ for $i=1,3,5$, and if we substitute these values of $u\left(\frac{i-1}{2}\right)_5$ into \eqref{eq:partitions_def_p}, we obtain
\begin{equation}
\label{eq:partitions_def_1}
s(n,k)_5=\sum_{\lambda\in \mathcal{P}^{3}_{2n,2k}}\left[\frac{(2n)!}{c_1!3!^{c_3}c_3!5!^{c_5}c_5!}\right]_5,
\end{equation}
with the convention that $s(n,k)_5=0$ if  $\mathcal{P}^{3}_{2n,2k}= \emptyset$.
The expression is already interesting. One immediate implication is that if $5k<n,$ then $r(n,k)_5=s(n,k)_5=0$, since $\mathcal{P}^{3}_{2n,2k}$ is clearly empty (see Figure 1). 

To reduce the sum in \eqref {eq:partitions_def_1} further, we recall that in the field of residues modulo 5, nonzero elements are invertible; therefore, since we've shown that each summand in \eqref {eq:partitions_def_1} is an integer, we can replace $3!$ in the denominator with $1$ and replace $5!=5\cdot 4!$ with $5\cdot(-1)$ without changing the value of the summand's residue mod 5. Thus, we have
\begin{equation}\label{eq:partitions_def_1.5}
s(n,k)_5=\sum_{\lambda\in P^{3}_{2n,2k}}\left[\frac{(2n)!(-1)^{c_5}}{c_1!c_3!c_5!}\right]_5.
\end{equation}

Next, identify elements of $P^{3}_{2n,2k}$ in the obvious way with triples $(c_1,c_3,c_5)$ of non-negative integers satisfying the pair of equations
$$\begin{cases}
\sum_{i=1}^3 ic_i&=2n\\
\sum_{i=1}^3 c_i&=2k.
\end{cases}
$$
For a given pair $(n,k)$, if we fix $c_5$ to be some integer $c$, then this becomes an invertible linear system with
\[
c_1=3k-n+c\, , \quad c_3=n-k-2c.
\] 
There exists $(c_1,c_3,c)\in P^{3}_{2n,2k}$ satisfying the system if and only if $n\leq5k$ and 
\[
\text{max}(0, n-3k)\leq c \leq \left\lfloor \frac{n-k}{2} \right\rfloor.\]
This allows us to rewrite \eqref{eq:partitions_def_1.5} as a summation over a single index parameter:
\begin{equation}\label{partitions_def_2}
s(n,k)_5=
\begin{cases}
\sum\limits_{c=\text{max}(0,n-3k)}^{\lfloor \frac{n-k}{2} \rfloor} \left[\frac{(2n)!(-1)^c}{
(3k-n+c)!(n-k-2c)!c!5^c}\right
]_5 &\, \text{ if }\, n\leq5k\\[1.4em]
0 &\, \text{ if } n>5k.
\end{cases}
\end{equation}

Our next step in the proof is to simplify \eqref{partitions_def_2} further by showing that the summation depends only on the term corresponding to the largest value of the index parameter, namely $c=\lfloor \frac{n-k}{2} \rfloor$, because all other terms are congruent to $0$ mod 5. This will be the content of Lemma \ref{oneterm} below. The proof of the lemma will use the following formula of Legendre \cite{M}, which we record here as a theorem. 
For $p$ a prime number, and $n$ a positive integer, let $\omega_p(n)$ denote the $p$-adic valuation of $n$, meaning that $\omega_p(n)$ is the largest natural number $\alpha$ such that $p^\alpha$ divides $n$. 
\begin{theorem}[Legendre]\label{legendre}
For any positive integer $n$, 
\[
\omega_p(n!)=\frac{n-s_p(n)}{p-1},
\]
where $s_p(n)$ is the sum of the digits in the base-p expansion of $n$. 
\end{theorem}

\begin{lemma}\label{oneterm}
For integers $0<k\leq n\leq 5k$, the quantity
\[
V(c):=\omega_5\left(\frac{(2n)!(-1)^c}{
(3k-n+c)!(n-k-2c)!c!5^c}\right),
\]
as a function of $c\in\mathbb{Z}$, is minimized over $\text{max}(0,n-3k)\leq c \leq \lfloor \frac{n-k}{2} \rfloor$ when $c=\lfloor \frac{n-k}{2} \rfloor$ and for no other values of $c.$
\end{lemma}
\begin{proof}
\begin{sloppypar}Assume $k+1<n<5k-1,$ as otherwise there is nothing to check. Let ${c\in \{\text{max}(0,n-3k), \cdots, \lfloor \frac{n-k}{2} \rfloor-1\}}$, and
let $\delta=\lfloor\frac{n-k}{2}\rfloor - c>0.$ Then,
\end{sloppypar} 
\begin{align}
\label{eq:val}
V(c)-V\left(\left\lfloor\frac{n-k}{2}\right\rfloor\right)&=V(c)-V(c+\delta)\notag\\
&=\omega_5((3k-n+c+\delta)!)-\omega_5((3k-n+c)!)\notag\\
&\,\,+\omega_5((n-k-2(c+\delta))!)-\omega_5((n-k-2c)!)\notag\\
&\,\,+\omega_5((c+\delta)!)-\omega_5(c!)\notag\\
&\,\,+\omega_5(5^{c+\delta})-\omega_5(5^c).\notag\\
\end{align}
Each line of the summation contains a difference that we would like to estimate from below. To do that, we note the
general fact that if $a$ and $b$ are positive integers, then
\begin{align*}\label{small-lemma}
&\omega_5((a+b)!)\,=\,\omega_5(a!)+\omega_5(b!)+\omega_5\left ({a+b \choose a}\right);\, \text{ hence,  } \\
&\omega_5((a+b)!)-\omega_5(a!)\,\geq\, \omega_5(b!).
\end{align*}
\noindent We also note that $n-k-2(c+\delta)=n-k-2\lfloor \frac{n-k}{2}\rfloor \in \{0,1\}$
and $n-k-2c\in \{2\delta, 2\delta +1\}.$
Therefore, we can bound from below each line in \eqref{eq:val} to obtain the estimate
\begin{align*}
V(c)-V\left(\left\lfloor\frac{n-k}{2}\right\rfloor\right)&\geq 2\omega_5(\delta!)-\omega_5((2\delta +1)!)+\delta.
\end{align*}

An application of Theorem \ref{legendre} now yields
\begin{align*}
V(c)-V\left(\left\lfloor\frac{n-k}{2}\right\rfloor\right)&\geq 2\frac{\delta-s_5(\delta)}{4}-\frac{2\delta +1 - s_5(2\delta+1)}{4}+\delta\\
&=\delta -\frac{s_5(\delta)}{2}+\frac{1}{4}(s_5(2\delta+1)-1)\\
&> \delta-s_5(\delta)\\
&\geq0, 
\end{align*}
where the last two inequalities amount to the simple fact that for $p$ prime, any integer $k>1$ satisfies $1< s_p(k)\leq k$.
We've shown that $V(c)$ assumes its smallest value uniquely at $c=\lfloor \frac{n-k}{2}\rfloor.$
\end{proof}
\noindent \emph {Proof of Theorem \ref{thm:formula}, continued.}\,\,
By the lemma, all of the summands in \eqref{partitions_def_2}, except the one indexed by $c=\lfloor \frac{n-k}{2}\rfloor$, must vanish mod 5, since they have positive valuation. The remaining summand may or may not vanish. In any case, we have the following simplified formula for $s(n,k)_5$, $1\leq k\leq n\leq 5k.$
\begin{equation}
\label{eq:s-cases}
s(n,k)\equiv
\begin{cases}
\frac{(2n)!(-1)^{\frac{n-k}{2}}}{\left(\frac{5k-n}{2}\right)! \left(\frac{n-k}{2}\right)!5^\frac{n-k}{2} }& \text{ if }\, n-k \text{ is even}\\[2em]
\frac{(2n)!(-1)^{\frac{n-k-1}{2}}}{\left(\frac{5k-n-1}{2}\right)! \left(\frac{n-k-1}{2}\right)!5^\frac{n-k-1}{2}}& \text{ if }\, n-k \text{ is odd}\\
\end{cases}
\end{equation}

Now we want to translate this into a formula for $r(n,k)_5=2^{n-k}_5s(n,k)_5$. The congruence of $(n-k)$ modulo 4 determines the congruence of $2^{n-k}$ modulo 5, as well as the sign of $(-1)^\frac{n-k}{2}$ (respectively $(-1)^\frac{n-k-1}{2}$) in the case $n-k$ is even (respectively odd). However, it turns out that we need only consider parity, since one can check routinely that
\begin{align*}
2^{n-k}(-1)^{\frac{n-k}{2}}\,\,&\equiv \,1\, \,\,\,\text{ if } n-k  \text{ is even,}\,\, \text{ and }\\
2^{n-k}(-1)^{\frac{n-k-1}{2}}&\equiv \,2 \, \,\,\,\text{ if } n-k \text{ is odd}.\\
\end{align*}
Combined with \eqref{eq:s-cases}, this completes the proof of Theorem \ref{thm:formula}. $\hfill\square$

\subsection{Proof of Theorem \ref{main_theorem} (ii)}
Now that we have a nice expression for $r(n,k)_5$, we return to the main objective of this section, proving Theorem \ref{main_theorem} (ii).
\begin {lemma}\label{lem:step}
In order to prove Theorem \ref{main_theorem} (ii), it suffices to prove the following:
For $n\geq 3$, 
\begin{equation}\label{eq:even-odd}
\sum_{\substack{\frac{n}{5}\leq k\leq n\\
k \text{ even}}}r(n,k)\equiv \sum_{\substack{\frac{n}{5}\leq k\leq n\\
k \text{ odd}}}r(n,k)\equiv 0.
\end{equation}
\end{lemma}
\begin{proof}
Assume that \eqref{eq:even-odd} holds. 
Then 
\[
\sum_{\frac{n}{5}\leq k\leq n}r(n,k)(-1)^k= \sum_{\substack{\frac{n}{5}\leq k\leq n\\
k \text{ even}}}r(n,k)-\sum_{\substack{\frac{n}{5}\leq k\leq n\\
k \text{ odd}}}r(n,k)\equiv 0.
\]
Subtracting $r(n,n)(-1)^n$ from the left and right sides, we obtain
\begin{equation}\label{eq:step1}
\sum_{\frac{n}{5}\leq k\leq n-1}r(n,k)(-1)^k\equiv r(n,n)(-1)^{n+1}.
\end{equation}
Now a quick application of Theorem \ref{thm:formula} shows that $r(n,n)\equiv 1$ for all $n$ (in fact, it's not hard to deduce from \eqref{eq:r} and the fact that $u(1)=1$ that $r(n,n)=1$ for all $n$), and we have also observed above that $r(n,k)\equiv 0$ when $5k<n$. Therefore, from \eqref{eq:step1} we obtain 
\begin{equation}\label{eq:step2}
\sum_{k=1}^{n-1}r(n,k)(-1)^k\equiv (-1)^{n+1}.
\end{equation}
We will now prove by induction that $d(n)\equiv (-1)^{n+1}$ for $n\geq 1$.
The cases $n=1$ and $n=2$ can be checked directly, since $d(1)=-1$ and $d(2)=51.$
Also from \eqref{eq:recurrence_d} and Lemma \ref{lemma:uv}, we see that when $n\geq 3$, the following holds:
\[
d(n)\equiv - \sum_{k=1}^{n-1}r(n,k)d(k).
\]
Thus, if $n\geq 3$ and we assume the induction hypothesis that $d(k)\equiv (-1)^{k+1}$ for all $1\leq k <n$, it follows that
\[
d(n)\equiv - \sum_{k=1}^{n-1}r(n,k)(-1)^{k+1}\equiv \sum_{k=1}^{n-1}r(n,k)(-1)^{k}.
\]
But the right-hand-side is congruent mod 5 to $(-1)^{n+1}$, by \eqref{eq:step2}. This verifies the induction step. 
Thus, the truth of \eqref{eq:even-odd} implies Theorem \ref{main_theorem} (ii). 
\end{proof}

We will now use some concepts from group theory to verify that \eqref{eq:even-odd} holds. For $n$ a positive integer, let $S_n$ denote the symmetric group on $n$ letters, and recall that every element of $S_n$ has a unique decomposition as a product of disjoint cycles. Let $X_n$ be the set of elements $x\in S_n$ 
such that $x^5=1$.
For any non-negative integer $k\leq n$, let $X_n^k$ denote the set of elements $x\in S_n$ such that $x$ can be written as a disjoint product of $k$ five-cycles and $n-5k$ one-cycles. Then
\begin{equation}\label{eq:union}
X_n=\bigcup_{k=0}^{\lfloor \frac{n}{5}\rfloor} X_n^k~.
\end{equation}
The connection to Theorem \ref{thm:formula} is the following lemma. 
\begin{lemma}\label{lem:lem} For $n>3$, 
\begin{align*}
&(i)& |X_{2n}|&=\sum_{\substack{\frac{n}{5}\leq k\leq n\\
 n-k \text{ even}}} \frac{(2n)!}{\left(\frac{5k-n}{2}\right)! \left(\frac{n-k}{2}\right)!5^{\frac{n-k}{2}}}~,\\[2em]
&(ii)&2(2n)(2n-1)(2n-2)\cdot \,|X_{2n-3}|&=\sum_{\substack{\frac{n}{5}\leq k<n\\ n-k \text{ odd}}} \frac{2(2n)!}{\left(\frac{5k-n-1}{2}\right)! \left(\frac{n-k-1}{2}\right)!5^{\frac{n-k-1}{2}}}~.
\end{align*}
\end{lemma}
\begin{proof}
\begin{sloppypar} 
Fix $n>3$. For each $k$, $0\leq k\leq n$, $X_n^k$ is a conjugacy class in $S_n$ with cardinality
\begin{equation*}\label{eq:card-1}
|X_n^k|=\frac{n!}{(n-5k)!k!5^k}
\end{equation*}
(see e.g. \cite[Prop. 11 and Exercise 33 in Sec. 4.3]{Dummit}),
and from \eqref{eq:union}, \[|X_{2n}|=\sum_{0\leq k \leq \frac{2n}{5}}|X_{2n}^k|
=\sum_{0\leq k \leq \frac{2n}{5}}\frac{(2n)!}{(2n-5k)!k!5^k}
.\]
Therefore, to prove part (i) of the lemma, we must show that the quantity $\frac{n-k}{2}$ assumes every value in the set $T_1=\{0,1,\dots, \lfloor \frac{2n}{5}\rfloor\}$ exactly once as $k$ ranges over the set {\sloppy$T_2=\{k:\lceil\frac{n}{5}\rceil\leq k\leq n, \, n-k \text{ even}\}$.}
This is not hard to see, since the change of variable $k\mapsto \frac{n-k}{2}$ maps $n$ to $0$, and is linear with first difference $-1/2$, while both $T_1$ and $T_2$ have the same cardinality, as one can deduce from a simple analysis of the cases of the congruence  mod $5$ of $n$.\end{sloppypar} 

Similarly,
 \[|X_{2n-3}|=\sum_{0\leq k \leq \frac{2n-3}{5}}|X_{2n-3}^k|=
 \sum_{0\leq k \leq \frac{2n-3}{5}}\frac{(2n-3)!}{(2n-3-5k)!k!5^k},
 \]
so to prove part (ii) we must show that the quantity $\frac{n-k-1}{2}$ assumes every value in the set 
$\{0,1,\dots, \lfloor \frac{2n-3}{5}\rfloor\}$ exactly once as $k$ ranges over $\{k:\lceil\frac{n}{5}\rceil\leq k\leq n-1, \, n-k \text{ odd}\}.$ This can be deduced from the change of variables $k\mapsto \frac{n-k-1}{2}$ and the same type of argument as before. 
\end{proof}
Lemma \ref{lem:lem} and Theorem \ref{thm:formula} together imply that if $n>3$ is even, then $\displaystyle{|X_{2n}|\equiv \sum_{\substack{\frac{n}{5}\leq k\leq n\\
k \text{ even}}}r(n,k)}$, and an integer multiple of $|X_{2n-3}|$ is congruent mod $5$ to
$\displaystyle{\sum_{\substack{\frac{n}{5}\leq k\leq n\\
k \text{ odd}}}r(n,k)}$. Therefore, in order to verify that \eqref{eq:even-odd} holds for $n>3$ even, it suffices to show that $|X_{2n}|\equiv|X_{2n-3}|\equiv 0$. This follows from a theorem of Frobenius (see e.g. \cite{F}).
\begin{theorem}[Frobenius]\label{frobenius} 
Let $G$ be a finite group whose order is divisible by a positive integer $m$. Then $m$ divides the cardinality of the set of solutions $x$ in $G$ to the equation $x^m=1$.
\end{theorem}
Since $X_n$ is precisely the set of solutions to the equation $x^5=1$ in $S_n$, the theorem implies that $|X_{2n}|\equiv |X_{2n-3}|\equiv 0$ for even $n>3$.

Similarly, if $n>3$ is odd, 
then $\displaystyle{|X_{2n}|\equiv \sum_{\substack{\frac{n}{5}\leq k\leq n\\
k \text{ odd}}}r(n,k)}$, and an integer multiple of $|X_{2n-3}|$ is congruent mod 5 to
$\displaystyle{\sum_{\substack{\frac{n}{5}\leq k\leq n\\
k \text{ even}}}r(n,k)}$, and we again apply Theoerem \ref{frobenius} to verify \eqref{eq:even-odd}. Finally, if $n=3$ we can check the validity of \eqref{eq:even-odd} by directly computing from $\eqref{eq:r-cases}$ that $r(3,1)\equiv 4, r(3,2)\equiv 0,$ and $r(3,3)\equiv1$.
This verifies \eqref{eq:even-odd} for all $n\geq 3$ and finishes the proof of Theorem \ref{main_theorem} (ii).

\section{Vanishing of $d(n)$ modulo primes $p=4k+3$}
\label{section:3mod4}
\subsection{A vanishing theorem for $u(n)$ and $v(n)$}
Throughout Section \ref{section:3mod4}, $p$ will always denote a prime congruent to $3$ mod $4$, $A\equiv B$ will be shorthand for $A\equiv B$ (mod $p$), and we define $n_0:=\frac{p^2-1}{2}.$
We begin with a theorem about the congruences of $(u(n))$ and $(v(n))$ modulo $p$, similar to Lemma \ref{lemma:uv} above.
\begin{theorem}
\label{thm:uv}
The sequences $(u(n))_{n=0}^\infty$ and $(v(n))_{n=0}^\infty$ satisfy
\begin{itemize}
\item[(i)]
\[u\left(\frac{p-1}{2}\right)\equiv 0,\]
\item[(ii)]
\[u(n)\equiv 0 \text { for }n\geq n_0,\]
\item[(iii)]
\[v(n)\equiv 0 \text { for }n > n_0.\]
\end{itemize}
\end{theorem}
We first prove a lemma that will be used repeatedly, then we prove the theorem in three parts. 
\begin{lemma}
\label{lem:tool}
If $a,b\in\mathbb{Z}$ and $p^2\leq a\leq b+p^2-1\leq 2p^2-2$, then ${a\choose b}\equiv 0$. 
\end{lemma}
\begin{proof}
The hypothesis implies that $b\leq p^2-1$ and $a-b\leq p^2-1$.
It follows that 
\[\omega_p(b!(a-b)!)=\left\lfloor \frac{b}{p} \right\rfloor +\left\lfloor \frac{a-b}{p}\right\rfloor \leq \frac{a}{p}.\]
Meanwhile, since $a \geq p^2$
\[
\omega_p(a!)\geq \left\lfloor \frac{a}{p} \right\rfloor +1>\frac{a}{p},
\]
and therefore $\omega_p\left({a\choose b}\right) >0.$
\end{proof}

\begin{proof}[Proof of Theorem \ref{thm:uv} (i)]
By \eqref{eq:u_recursion}, we have
\begin{align}
\label{eq:vanishing}
u\left(\frac{p-1}{2}\right)&=(3\cdot 7\cdots (2p-3))^2 \notag\\
&-\sum_{m=0}^{\frac{p-1}{2}-1}{p\choose 2m+1}\left[1\cdot 5\cdots \left(4\left(\frac{p-1}{2}-m\right)-3\right)\right]^2u(m).
\end{align}
The product $(3\cdot 7\cdots (2p-3))^2$ contains as factors all positive integers that are congruent to $3$ mod $4$ and less than $2p+1$, and $p$ is such a number. Furthermore ${p\choose m}\equiv 0$ for $1\leq m <p$, so the sum in \eqref{eq:vanishing} also vanishes mod $p$. 
\end{proof}

\begin{proof}[Proof of Theorem \ref{thm:uv} (ii)]
Set $n_1=\frac{3(p+1)}{4}<n_0.$ 
Referring to \eqref{eq:u_recursion},
observe that $3\cdot 7\cdots (4n-1))^2\equiv 0$ for $n\geq\frac{p+1}{4}$, so in particular for $n\geq n_0$.
Observe also that 
$
0\equiv 1\cdot 5\cdots (4(n-m)-3)
$
if $n-m\geq n_1$.
It follows that if $n\geq n_0$, we have the following truncated summation for $u(n)$: 
\begin{equation}
\label{eq:u-truncated}
u(n)\equiv \sum_{m=n-n_1+1}^{n-1}{2n+1 \choose 2m+1} (1\cdot 5\cdots (4(n-m)-3))^2u(m).
\end{equation}

We will also use the fact that
\begin{equation}
\label{eq:tool}
{2n+1\choose 2m+1}\equiv 0 
\end{equation}
$
\text{for } \,\,n_0\leq n\leq n_0+n_1 -2 \,\,\text{ and } \,\,n_0-n_1+1 \leq m \leq n_0-1,
$
which follows from Lemma \ref{lem:tool}. Indeed, the assumptions on $n$ in \eqref{eq:tool} imply that 
\[
p^2=2n_0+1\leq 2n+1\leq 2n_0+2n_1-3\leq 2p^2-2, 
\]
and hence
\[
p^2+\frac{3}{2}(p+1)+2= 2n_0-2n_1+3\leq 2m+1\leq 2n_0-1=p^2-2.
\]
But $p^2+\frac{3}{2}(p+1)+2\geq (2n+1)-p^2+1$, for $n\leq n_0+n_1-2.$
In brief, Lemma \ref{lem:tool} applies with $a=2n+1$ and $b=2m+1$, verifying \eqref{eq:tool}.

It follows that $u(n_0)\equiv 0$, since \eqref{eq:tool} implies that all of the binomial coefficients in \eqref{eq:u-truncated} vanish mod $p$ when $n=n_0$. Now suppose that \[u(n_0)\equiv u(n_0+1)\equiv \cdots \equiv u(n_0+k-1)\equiv 0,\] for some $k$ such that $1\leq k\leq n_1-2$. This supposition, along with \eqref{eq:u-truncated}, implies that
\begin{align*}
u(n_0+k)&\equiv\sum_{m=n_0+k-n_1+1}^{n_0+k-1}\bigg{[}{2(n_0+k)+1 \choose 2m+1}\\
&\quad\quad\quad\quad\quad\quad\quad\times(1\cdot 5\cdots (4(n_0+k-m)-3))^2u(m)\bigg{]}\\
&\equiv\sum_{m=n_0+k-n_1+1}^{n_0-1} \bigg{[}{2(n_0+k)+1 \choose 2m+1}\\
&\quad\quad\quad\quad\quad\quad\quad\times(1\cdot 5\cdots (4(n_0+k-m)-3))^2u(m)\bigg{].}
\end{align*}
By \eqref{eq:tool} all the binomial coefficients in the sum vanish; hence $u(n_0+k)\equiv 0$. Since $k$ was arbitrary,  we can conclude that $u(n)\equiv 0$ when $n_0\leq n\leq n_0+n_1-2$.

Finally, by \eqref{eq:u-truncated} $u(n)_p$ is a sum involving only those values of $u$ evaluated at integers in $[n-n_1+1, n-1]$; if these values of $u$ vanish mod $p$, then so does $u(n)$. Therefore, if one can show that $u(n)$ vanishes mod $p$ for $n_1-1$ consecutive values of $n$, then by induction $u(n)$ must vanish mod $p$ for all larger $n$.
But we've already shown above that $u(n)$ vanishes for $n\in[n_0,n_0+n_1-2],$ so it follows that $u(n)$ vanishes for all $n\geq n_0$.\end{proof}

\begin{proof}[Proof of Theorem \ref{thm:uv} (iii)]
Referring to the recursive definition for $v(n)$ in \eqref{eq:v_recursion}, observe that $1\cdot 5\cdots (4n-3)\equiv 0$ for $n\geq \frac{3p+3}{4}$ and hence for $n\geq n_0$.
Furthermore, if $n=n_0+1$ and $1\leq m\leq n_0$, then Lemma \ref{lem:tool} applies with $a=2n$ and $b=2m$ and hence ${2n\choose 2m}\equiv 0$. Therefore, $v(n_0+1)\equiv 0$.

Now let $n>n_0$ be arbitrary, and assume as an induction hypothesis that $v(k)\equiv 0$ for all $n_0< k< n$.
Then
\begin{equation}
\label{eq:v-truncated}
v(n)\equiv -\frac{1}{2}\sum_{m=1}^{n-1} {2n\choose 2m} v(m)v(n-m).
\end{equation}
If $n\geq 2n_0$, then $n-m>n_0$ for all values of the summation index $m$, so by the induction hypothesis $v(n-m)$ vanishes mod $p$ and so does the sum. So assume that $n\leq 2n_0$. Then we may restrict the sum in \eqref{eq:v-truncated} to index values $m\in[n-n_0,n_0]$, since for other values of $m$ either $m>n_0$ or $n-m>n_0$.
However, for any such $m$, we can apply Lemma \ref{lem:tool} with $a=2n$ and $b=2m$, since $p^2\leq 2n$
and $2n-p^2+1\leq 2m\leq 2n_0=p^2-1$. It follows that every binomial coefficient in $\eqref{eq:v-truncated}$ vanishes mod $p$ and so does $v(n)$. Induction on $n>n_0$ completes the proof. 
\end{proof}
\subsection{Proof of Theorem \ref{main_theorem} (iii)}
We begin with a lemma that provides a means for proving the theorem.
\begin{lemma}\label{lem:sufficient}
If $r(n,k)\equiv 0$ for all pairs $(n,k)$ such that $1\leq k\leq n_0<n$, then Theorem \ref{main_theorem} (iii) is true. 
\end{lemma}
Figure 2 below gives an illustration of the lemma's hypothesis in the case $p=7$. 

\begin{figure*}[h]
\center{\includegraphics[scale=0.65]{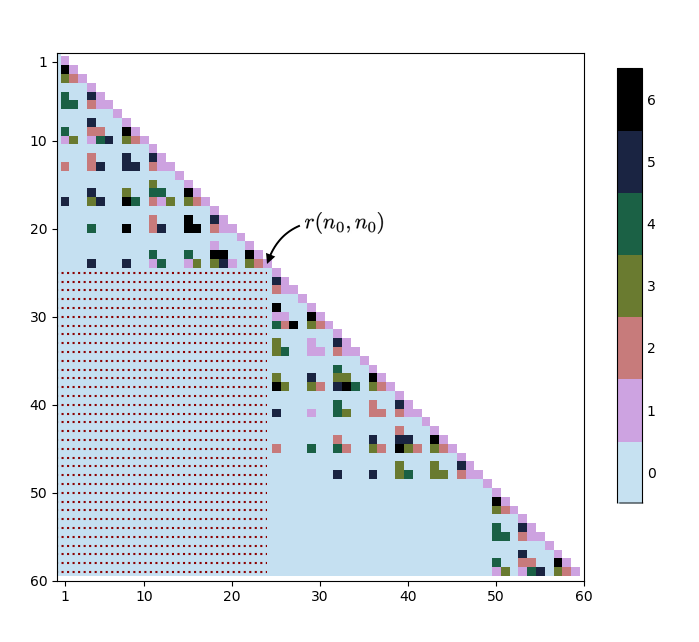}}
\caption{Congruences of $r(n,k)$ mod 7, $1\leq k\leq n< 60$. The rows are indexed by $n$, and the colors indicate residue classes of $r(n,k)$ mod 7, according to the colorbar. Note that $n_0=\frac{p^2-1}{2}=24$ for $p=7$. The submatrix $(r(n,k))_{1\leq k\leq n_0<n}$, where $r(n,k)$ vanishes by Theorem \ref{thm:valuation}, is emphasized.}
\end{figure*}

\begin{proof}
Equation \eqref{eq:recurrence_d} and Theorem \ref{thm:uv} imply that if $n>n_0$ then
\begin{equation*}
d(n)\equiv -\sum_{k=1}^{n-1} r(n,k)d(k).
\end{equation*}
If we assume the hypothesis of the lemma, then 
\begin{equation*}
d(n_0+1)\equiv -\sum_{k=1}^{n_0} r(n_0+1,k)d(k)\equiv 0, 
\end{equation*}
since all the summands vanish mod $p$; moreover for general $n>n_0$,
\begin{equation*}
d(n+1)\equiv -\sum_{k=n_0+1}^{n} r(n,k)d(k).
\end{equation*}
Therefore, if we assume that $d(k)\equiv 0$ for all $k$ such that $n_0+1\leq k\leq n$, then $d(n+1)\equiv 0$. It follows by induction that $d(n)\equiv 0$ for all $n>n_0$, which is the statement of Theorem \ref{main_theorem} (iii).
\end{proof}
As we did in Section \ref{section:p=5}, for $p=5$, we would now like to restrict the class of partitions that we need to consider in  
\eqref{eq:partitions_def_p}, for $p\equiv 3$ (mod $4$). Let $\mathcal{P}^*_{2n,2k}\subset\mathcal{P}'_{2n,2k}$ denote the set of partitions $\lambda$ whose parts are all less than $p^2$, and such that no part of $\lambda$ is equal to $p$, i.e. $c_p=0$. 
Theorem \ref{thm:uv} implies that we may replace the index set in the summation \eqref{eq:partitions_def_p} with $\mathcal{P}^*_{2n,2k}$, since any partition $\lambda\in \mathcal{P}'_{2n,2k}\setminus{\mathcal{P}^*_{2n,2k}}$ must contain a part $\lambda_i$ such that $u\left(\frac{\lambda_i-1}{2}\right)_p=0$
and hence will contribute $0$ to the sum. 
In other words, 
\begin{equation}
\label{eq:partitions_def_p_2}
s(n,k)_p=\sum_{\lambda\in \mathcal{P}^*_{2n,2k}}\left( \left[\frac{(2n)!}{\prod_{i=1}^{2n} i!^{c_i}c_i!}\right]_p \prod_{i=1}^{2n}\left[u\left(\frac{i-1}{2}\right)^{c_i}\right]_p\right),
\end{equation}
with the convention that $s(n,k)_p=0$ if  $\mathcal{P}^{*}_{2n,2k}= \emptyset$.
The key to using formula $\eqref{eq:partitions_def_p_2}$ is the following theorem. 
\begin{theorem}\label{thm:valuation}
Let $(n,k)$ be such that $1\leq k\leq n_0<n.$ Let $\lambda=(c_1,c_2,\dots, c_{2n})$ be a partition in $\mathcal{P}^*_{2n,2k}.$ Then,
\begin{equation}
\omega_p\left(\frac{(2n)!}{\prod_{i=1}^{2n}i!^{c_i}c_i!}\right)>0.
\end{equation}
\end{theorem}
The theorem implies,
by \eqref{eq:partitions_def_p_2},
that $r(n,k)=2^{n-k}s(n,k)$ vanishes mod $p$ for $1\leq k\leq n_0<n$, and in view of Lemma \ref{lem:sufficient} will complete the proof of Theorem \ref{main_theorem} (iii). 

We record as a lemma a few facts about arithmetic that will be used freely in the proof of Theorem \ref{thm:valuation}.
\begin{lemma}\label{facts}
Let $r$ and $s$ be positive integers with base-$p$ expansions \[r=\sum_{i=1}^\infty r_ip^i\,\,\, \text{ and }\,\,\, s=\sum_{i=1}^\infty s_ip^i.\]
Then the following are true:
\begin{itemize}
\item[(i)]
$\displaystyle{s_p(r+s)\leq s_p(r)+s_p(s)}$, with equality iff there are no carries when $r$ is added to $s$ in base $p$, iff $r_i+s_i\leq p-1$ for all $i$.
\item[(ii)]
$\displaystyle{\sum_{i=1}^\infty(r_i+s_i)\geq s_p\left[\sum_{i=1}^\infty (r_i+ps_i)\right ]}$
\item[(iii)]
$s_p(rp)=s_p(r)$ always, and $s_p(r)=r$ iff $r\leq p-1$.
\end{itemize}
\end{lemma}
\begin{proof}
Statements (i) and (iii) are trivial. We use them to verify (ii).
\begin{align*}
\sum_{i=1}^\infty(r_i+s_i)&\geq s_p\left(\sum_{i=1}^\infty r_i\right)+s_p\left(\sum_{i=1}^\infty s_i\right)\\
&=s_p\left(\sum_{i=1}^\infty r_i\right)+s_p\left(p\sum_{i=1}^\infty s_i\right)\\
&\geq s_p\left(\sum_{i=1}^\infty r_i +p \sum_{i=1}^\infty s_i\right)\\
&=s_p\left[\sum_{i=1}^\infty (r_i+ps_i)\right ].
\end{align*}
\end{proof}

\begin{proof}[Proof of Theorem \ref{thm:valuation}]
Let $\lambda=(c_1,c_2,\dots, c_{2n})\in \mathcal{P}^*_{2n,2k}.$ For each $i$, $1\leq i \leq 2n$, let $c_i=a_0^i+a_1^i p$ be the base-$p$ expansion of $c_i$, and let $i=b_0^i+b_1^ip$ be the base-$p$ expansion of $i$. (The exponents are indices, and the fact that there are at most two digits in each expansion follows from the conditions $k\leq n_0$ and $\lambda_i< p^2$ for all $i$.)

By Theorem \ref{legendre},
\begin{align*}
\omega_p\left(\frac{(2n)!}{\prod_{i=1}^{2n} i!^{c_i}c_i!}\right)(p-1)&=2n-s_p(2n)-\sum_{i=1}^{2n} [c_i(i-s_p(i))+(c_i-s_p(c_i))\notag
]\\
&=\sum_{i=1}^{2n}s_p(i)c_i+\left[\sum_{i=1}^{2n}(s_p(c_i)-c_i)\right]-s_p(2n)\notag,
\end{align*}
where the last equality comes from the fact that $2n=\sum_{i=1}^{2n} ic_i.$
We expand all the terms in the last line base-$p$, obtaining
\begin{align*}
&\omega_p\left(\frac{(2n)!}{\prod_{i=1}^{2n} i!^{c_i}c_i!}\right)(p-1)\notag\\
&=\sum_{i=1}^{2n}(b_0^i+b_1^i)(a_0^i+a_1^ip)+\sum_{i=1}^{2n}a_1^i(1-p)-s_p\left(\sum_{i=1}^{2n}(b_0^i+b_1^ip)(a_0^i+a_1^ip)\right)\notag\\
\end{align*}
\begin{align}
&=\sum_{i=1}^{2n}(b_0^i+b_1^i)a_1^ip+\sum_{i=1}^{2n}(b_0^i+b_1^i)a_0^i+\sum_{i=1}^{2n}a_1^i(1-p)\notag\\
&\quad\quad\quad\quad\quad\quad\quad\,\,\,\,\,- s_p\left(\sum_{i=1}^{2n}(b_0^i+b_1^ip)a_0^i+\sum_{i=1}^{2n}(b_0^i+b_1^ip)a_1^ip\right)\label{eq:0}\\
&\geq \sum_{i=1}^{2n}(b_0^i+b_1^i)a_1^ip - s_p\left(\sum_{i=1}^{2n}(b_0^i+b_1^ip)a_1^ip\right)+\sum_{i=1}^{2n} a_1^i(1-p)\label{eq:1}\\
&\quad\quad\quad\quad\quad\quad\quad\,\,\,\,\,+\sum_{i=1}^{2n}(b_0^i+b_1^i)a_0^i-s_p\left(\sum_{i=1}^{2n}(b_0^i+b_1^ip)a_0^i\right)\label{eq:2}\\
&\geq 0,\notag
\end{align}
where we justify the last inequality as follows. The quantity in \eqref{eq:2} is non-negative by Lemma \ref{facts} (ii), and we claim that the quantity in \eqref{eq:1} is also non-negative. To show that, write the quantity as
\[
\left\{\left[\sum_{i=1}^{2n} (b_0^i+b_1^i)a_1^i\right] -s_p\left(\sum_{i=1}^{2n} (b_0^i+b_1^ip)a_1^ip\right)\right\}\]
\[+
\left\{\left[\sum_{i=1}^{2n} (b_0^i+b_1^i)a_1^i(p-1)\right] +\left[\sum_{i=1}^{2n}a_1^i(1-p)\right]\right\}.
\]
The first bracketed term is non-negative by Lemma \ref{facts}, and the second bracketed term is non-negative since $b_0^i+b_1^i\geq 1$ for all $i$.

Suppose now that $\omega_p\left(\frac{(2n)!}{\prod_{i=1}^{2n} i!^{c_i}c_i!}\right)=0$. Then the inequality that we used to transition from \eqref{eq:0} to \eqref{eq:1} and \eqref{eq:2} must be an equality, and the quantity in \eqref{eq:2} and the bracketed quantities above must all vanish. These facts have a number of consequences. First, from the second bracketed quantity, observe that for all $i$, 
\begin{equation*}
a_1^i(b_0^i+b_1^i)=a_1^i,
\end{equation*}
so either $a_1^i=0$ or $(b_0,b_1)=(1,0)$ or $(b_0,b_1)=(0,1).$ Since no part of $\lambda$ is equal to $p$, by the definition of $\mathcal{P}^*_{2n,2k}$, the last option is not possible. We conclude that $a_1^i=0$ for all $i\geq 2$. In view of this, the identity
\begin{align*}
&s_p\left(\sum_{i=1}^{2n}(b_0^i+b_1^ip)a_0^i + \sum_{i=1}^{2n}(b_0^i+b_1^ip)a_1^ip \right)\\
&=s_p\left(\sum_{i=1}^{2n}(b_0^i+b_1^ip)a_0^i \right)+ s_p\left(\sum_{i=1}^{2n}(b_0^i+b_1^ip)a_1^ip \right)
\end{align*}
(from the transition from \eqref{eq:0} to \eqref{eq:1} and \eqref{eq:2}), can be simplified to
\begin{equation}
\label{eq:4}
s_p\left(\sum_{i=1}^{2n}(b_0^i+b_1^ip)a_0^i + a_1^1p\right)=s_p\left(\sum_{i=1}^{2n}(b_0^i+b_1^ip)a_0^i \right)+ a_1^1.
\end{equation}
Furthermore, the fact that the quantity in \eqref{eq:2} vanishes implies that \eqref{eq:4} can be written as
\begin{equation}
\label{eq:5}
s_p\left(\sum_{i=1}^{2n}(b_0^i+b_1^ip)a_0^i + a_1^1p\right)=
\left[\sum_{i=1}^{2n}(b_0^i+b_1^i)a_0^i\right] + a_1^1.
\end{equation}
Since we can estimate the left-hand-side by
\begin{align}\label{eq:6}
s_p\left(\sum_{i=1}^{2n}(b_0^i+b_1^ip)a_0^i + a_1^1p\right)&\leq s_p\left(\sum_{i=1}^{2n}b_0^ia_0^i\right)+s_p\left (p\left[\sum_{i=1}^{2n}b_1^ia_0^i\right]+a_1^1p\right)\notag\\
&=s_p\left(\sum_{i=1}^{2n}b_0^ia_0^i\right)+s_p\left(\left[\sum_{i=1}^{2n}b_1^ia_0^i\right]+a_1^1\right)\notag\\
&\leq \left[\sum_{i=1}^{2n}b_0^ia_0^i\right] + \left[\sum_{i=1}^{2n}b_1^ia_0^i\right] +a_1^1,\notag\\
\end{align}
which is the right-hand-side of $\eqref{eq:5},$ we must have equality throughout \eqref{eq:6}.
It follows that
\[
s_p\left(\sum_{i=1}^{2n}b_0^ia_0^i\right)+s_p\left(\left[\sum_{i=1}^{2n}b_1^ia_0^i\right]+a_1^1\right)=\left[\sum_{i=1}^{2n}b_0^ia_0^i\right] + \left[\sum_{i=1}^{2n}b_1^ia_0^i\right] +a_1^1,
\]
and hence, by Lemma \ref{facts} (iii), we see that
\begin{equation}
\label{eq:doo}
\left[\sum_{i=1}^{2n}b_0^ia_0^i\right]\leq p-1 \,\,\text{ and }\,\,
\left[\sum_{i=1}^{2n}b_1^ia_0^i\right] +a_1^1 \leq p-1.
\end{equation}
Recalling that $2n=\sum_{i=1}^{2n} ic_i$, we see from \eqref{eq:doo} that 
\begin{align*}
2n&=a_1^1p+\sum_{i=1}^{2n}(b_0^i+b_1^ip)a_0^i\\
&=\sum_{i=1}^{2n}b_0^ia_0^i+p\left[\left(\sum_{i=1}^{2n} b_1^ia_0^i\right)+a_1^1\right]\\
&\leq (p-1)+p(p-1)\\
&=p^2-1.
\end{align*}
In conclusion, the supposition above that $\omega_p\left(\frac{(2n)!}{\prod_{i=1}^{2n} i!^{c_i}c_i!}\right)=0$ led us to the statement that $2n\leq p^2-1$. But the theorem assumed $n>n_0$, and hence $2n>p^2-1$. This contradiction completes the proof of Theorem \ref{thm:valuation}.
\end{proof}

\subsection{Closing remarks}
The inspiration for the proof above was supplied by Figure 2, from which the statement of Theorem \ref{thm:valuation} in the case $p=7$ seems obvious. 
It is now interesting and natural to ask why the results for $p=5$ and $p=4k+3$ are so different. Surely the difference should be reflected somewhere in the proofs of the respective statements. The most important difference is that $u(2)_5=u\left(\frac{5-1}{2}\right)_5=1$, whereas $u\left(\frac{p-1}{2}\right)_p=0$ for $p=4k+3$. Consequently, there exist partitions in $\mathcal{P}^3_{2n,2k}$ that index non-vanishing summands in \eqref{eq:partitions_def_p}, and indeed those partitions with the largest possible number of $5$'s are the important ones. On the other hand, when $p=4k+3,$ partitions in $\mathcal{P}'_{2n,2k}$ with parts equal to $p$ do not contribute to the sum in \eqref{eq:partitions_def_p} at all. It seems plausible, then, that a statement similar to Theorem \ref{thm:uv} that is general for powers $p^\alpha$ of primes $p=4k+3$ would exist and be useful in proving the experimentally evident conjecture from \cite{Romik} that $d(n)$ eventually vanishes mod $p^\alpha.$ 

\section*{Acknowledgments}
The author would like to especially thank Dan Romik for suggesting the topic of this paper, for sharing with me a version of Figure 1, and for many helpful consultations. The author would also like to thank an anonymous referee for suggesting improvements to an earlier version of this paper, and the author would like to thank Tanay Wakhare for suggesting an improvement to the earlier version's proof in Section \ref{section:odd}. This material is based upon work supported by the National Science Foundation under Grant No. DMS-1800725.


\end{document}